\documentclass[10pt,reqno]{amsart}

\usepackage{a4wide}
\usepackage{dsfont} 	   % provides \mathds font used in \I
\usepackage{mathrsfs}
\usepackage[hidelinks]{hyperref}

\newtheorem{proposition}{Proposition}[section]
  \newtheorem{theorem}[proposition]{Theorem}
  \newtheorem{lemma}[proposition]{Lemma}
  \newtheorem{corollary}[proposition]{Corollary}
\theoremstyle{definition}
  
  \newtheorem{remark}[proposition]{Remark}
  \newtheorem{example}[proposition]{Example}

\numberwithin{equation}{section}
\numberwithin{proposition}{section}

\DeclareMathOperator{\Mor}{Mor}
\DeclareMathOperator{\M}{\mathsf{M}}
\DeclareMathOperator{\C}{C}

\newcommand{\st}{\;\vline\;} 	% such that
\newcommand{\cst}{\ifmmode\mathrm{C}^*\else{$\mathrm{C}^*$}\fi}
\newcommand{\tens}{\otimes}
\newcommand{\id}{\mathrm{id}}
\newcommand{\eps}{\varepsilon}
\newcommand{\comp}{\circ}
\newcommand{\qqquad}{\quad\qquad}
\newcommand{\I}{\mathds{1}}

\newcommand{\CC}{\mathbb{C}}
\newcommand{\GG}{\mathbb{G}}
\newcommand{\KK}{\mathbb{K}}

\newcommand{\sC}{\mathsf{C}}
\newcommand{\sB}{\mathsf{B}}
\newcommand{\sA}{\mathsf{A}}

\newcommand{\cG}{\mathscr{G}}
\newcommand{\cE}{\mathscr{E}}

\begin{document}

\title{Quantum Semigroups from Synchronous Games}

\date{\today}

\author{Piotr M.~So{\l}tan}
\address{Department of Mathematical Methods in Physics\\
Faculty of Physics\\
Warsaw University}
\email{piotr.soltan@fuw.edu.pl}

% \thanks{Research partially supported by NCN (National Science Center, Poland) grant no.~2015/17/B/ST1/00085.}

\begin{abstract}
We show that the $\mathrm{C}^*$-algebras associated with synchronous games give rise to certain quantum families of maps between the input and output sets of the game. In particular situations (e.g.~for graph endomorphism games) these quantum families have a natural quantum semigroup structure and if the condition of preservation of a natural state is added, they are in fact compact quantum groups.
\end{abstract}

% 46L89  	Other "noncommutative'' mathematics based on $C^*$-algebra theory
% 58B32  	Geometry of quantum groups
% 58B34  	Noncommutative geometry (à la Connes)
% 81Q35  	Quantum mechanics on special spaces: manifolds, fractals, graphs, etc.
% 05C25  	Graphs and abstract algebra (groups, rings, fields, etc.)
% 05C60  	Isomorphism problems (reconstruction conjecture, etc.) and homomorphisms (subgraph embedding, etc.)
% 05C57  	Games on graphs
% 91A80  	Applications of game theory
% 91A43  	Games involving graphs


\maketitle

\section{Introduction}

Recently there has been a considerable amount of activity focused on certain extension of the notion of a homomorphism of graphs. The resulting notions of a \emph{quantum homomorphism} or \emph{quantum endomorphism} described e.g.~in \cite[Section 2]{MD} and \cite{HMPS} are based on the notion of a two-person game and various ``quantum strategies'' for such games (a more thorough description of this is contained in Example \ref{grHom} below). The latter paper provides a construction which associates a \cst-algebra to each synchronous game -- including the graph homomorphism game. The aim of this paper is to provide some broader context for these objects by relating them to the notion of a \emph{quantum family of maps} studied e.g.~in \cite{qfam,qmqs,invert}. Similar ideas have appeared recently in this context e.g.~in \cite{R1,R2,R3}.

In what follows we will briefly recall the definition of a \cst-algebra associated to a synchronous game, show what quantum family of maps it is related to and describe its universal property. Then we will identify certain special games for which the this quantum family of maps naturally defines a quantum semigroup.

In the last section we add a condition that the quantum families preserve a natural measure on the set of inputs of the game which forces the quantum semigroup to become a quantum group which is in fact a quantum subgroup of the quantum permutation group defined in \cite{wang}. This leads us to the notion of a quantum automorphism group of a finite graph introduced by T.~Banica in \cite{banica} and which is closely related to the one defined by J.~Bichon in \cite{jb}.

We will not be using much of the theory of compact quantum groups, but the interested reader will find all the necessary information e.g.~in \cite{cqg} as well as introductory sections of \cite{jb,wang}. Since quantum semigroups are less sophisticated, there does not seem to be ample literature dealing with these objects. The definitions and elementary examples are contained e.g.~in \cite{qfam,qmqs}. For the purposes of this paper it is enough to remember that a \emph{quantum semigroup} $\mathbb{S}$ is described by a \cst-algebra denoted $\C_0(\mathbb{S})$ and a morphism $\Delta_{\mathbb{S}}\in\Mor(\C_0(\mathbb{S}),\C_0(\mathbb{S})\tens\C_0(\mathbb{S}))$ (see below for a brief discussion of the notion of a morphism of \cst-algebras) called the \emph{comultiplication} which satisfies the condition of \emph{coassociativity}, i.e.
\[
(\Delta_{\mathbb{S}}\tens\id)\comp\Delta_{\mathbb{S}}=(\id\tens\Delta_{\mathbb{S}})\comp\Delta_{\mathbb{S}}.
\]
A quantum semigroup is \emph{compact} if the \cst-algebra $\C_0(\mathbb{S})$ is unital and in this case we write $\C(\mathbb{S})$ instead of $\C_0(\mathbb{S})$.

Throughout this note we will be using the language of the theory of \cst-algebras. The class of morphisms between \cst-algebras appropriate for ``non-commutative topology'' is the one proposed e.g.~in \cite{unbo}: if $\sA$ and $\sC$ are \cst-algebras then a morphism from a $\sA$ to $\sC$ is a $*$-homomorphism $\Phi:\sA\to\M(\sC)$ (where $\M(\sC)$ is the \emph{multiplier algebra} of $\sC$) which is non-degenerate, i.e.~$\Phi(\sA)\sC$ is dense in $\sC$. We will denote the set of all morphisms form $\sA$ to $\sC$ by the symbol $\Mor(\sA,\sC)$. In case $\sA$ is unital, elements of $\Mor(\sA,\sC)$ are simply unital $*$-homomorphisms form $\sA$ to $\M(\sC)$. In order to form compositions of morphisms of \cst-algebras one uses a natural extension of each $\Phi\in\Mor(\sA,\sC)$ to a $*$-homomorphism $\M(\sA)\to\M(\sC)$ (see \cite[Section 0]{unbo} or \cite[Chapter 2]{lance}).

We end this introduction with a simple and certainly well-known lemma concerning projections. By a \emph{projection} we mean an element $p$ of a \cst-algebra which satisfies $p^*p=p$.

\begin{lemma}\label{commpij}
Let $p_1,\dotsc,p_N$ be projections such that $p_1+\dotsm+p_N=\I$. Then for any $i,j\in\{1,\dotsc,N\}$ we have $p_ip_j=\delta_{i,j}p_i$.
\end{lemma}

\begin{proof}
Clearly it is enough to consider $i\neq{j}$. Let $q=\sum\limits_{i\neq{k}\neq{j}}p_k$. Then $q\geq{0}$ and $p_i+p_j+q=\I$. We have
\[
0\leq{p_jp_ip_j}=p_j(\I-p_j-q)p_j=p_j(\I-p_j)p_j-p_jqp_j=-p_jqp_j\leq{0}
\]
and hence $p_jp_ip_j=0$. It follows that $(p_ip_j)^*(p_ip_j)=p_j{p_i}^2p_j=p_jp_ip_j=0$ and consequently $p_ip_j=0$.
\end{proof}

It follows from Lemma \ref{commpij} that $\CC^N$ is the universal \cst-algebra generated by a collection of $N$ projections summing up to $\I$. Indeed, if $p_1+\dotsm+p_N=\I$ then any non-commutative polynomial in $p_1,\dotsc,p_N$ belongs to $\operatorname{span}\{p_1,\dotsc,p_N\}$ and hence the dimension of the universal \cst-algebra generated by $N$ projections summing up to $\I$ is not greater than $N$. In particular for any \cst-algebra $\sC$ a morphism $\Phi\in\Mor(\CC^N,\sC)$ is uniquely determined by a collection of $N$ projections $p_1,\dotsc,p_N\in\M(\sC)$ summing up to $\I$. We then have $\Phi(e_i)=p_i$, where $\{e_1,\dotsc,e_N\}$ is the standard basis of $\CC^N$.

\section{The \texorpdfstring{\cst}{C*}-algebra of a synchronous game}

A \emph{two-person finite input-output game} is specified by four finite sets $I_{\text{\rm{A}}},I_{\text{\rm{B}}},O_{\text{\rm{A}}},O_{\text{\rm{B}}}$ ($I$ and $O$ standing for ``input'' and ``output'' with subscripts $\text{\rm{A}}$ and $\text{\rm{B}}$ meaning ``Alice'' and ``Bob'') and a function
\[
\lambda:O_{\text{\rm{A}}}\times{O_{\text{\rm{B}}}}\times{I_{\text{\rm{A}}}}\times{I_{\text{\rm{B}}}}\longrightarrow\{0,1\}
\]
specifying the \emph{rules} of the game. The game is played by Alice and Bob with a referee who asks questions $x\in{I_{\text{\rm{A}}}}$ and $y\in{I_{\text{\rm{B}}}}$ of Alice and Bob and they independently provide answers $a\in{O_{\text{\rm{A}}}}$ and $b\in{O_{\text{\rm{B}}}}$. The value of $\lambda(a,b,x,y)$ is interpreted as the outcome of a round: $\lambda(a,b,x,y)=1$ means Alice and Bob win, while $\lambda(a,b,x,y)=0$ indicates they lose. In what follows we will use the notation $\cG=(I_{\text{\rm{A}}},I_{\text{\rm{B}}},O_{\text{\rm{A}}},O_{\text{\rm{B}}},\lambda)$ for a a given two-person finite input-output game.

A game $\cG=(I_{\text{\rm{A}}},I_{\text{\rm{B}}},O_{\text{\rm{A}}},O_{\text{\rm{B}}},\lambda)$ is called \emph{synchronous} if $I_{\text{\rm{A}}}=I_{\text{\rm{B}}}$, $O_{\text{\rm{A}}}=O_{\text{\rm{B}}}$ (denoted simply by $I$ and $O$) and $\lambda$ satisfies
\begin{equation}\label{lambdaSynch}
\lambda(a,b,x,x)=\delta_{a,b},\qqquad{x}\in{I},\:{a,b}\in{O}.
\end{equation}
One may propose the term ``easiest synchronous game'' for the game with $\lambda=1$ except in cases given by \eqref{lambdaSynch}. The terminology reflects the fact that it is ``easiest'' to win a game with $\lambda$ having maximally many values $1$. We will denote the easiest game with input set $I$ and output set $O$ by $\cE_{I,O}$.

Let $\cG=(I,O,\lambda)$ be a synchronous game. In \cite{HMPS} the \cst-algebra $\sA(\cG)$ is defined as the universal \cst-algebra generated by a family $\{p_{x,a}\}_{x\in{I},a\in{O}}$ of projections satisfying
\[
\sum_{a\in{I}}p_{x,a}=\I,\qqquad{x}\in{I}
\]
and
\[
p_{x,a}p_{y,b}=\lambda(a,b,x,y)p_{x,a}p_{y,b},\qqquad{x,y}\in{I},\:a,b\in{O}
\]
(in other words $p_{x,a}$ and $p_{y,b}$ are orthogonal whenever $\lambda(a,b,x,y)=0$).

In what follows we shall write $\CC^I$ and $\CC^O$ for the algebras of all functions $I\to\CC$ and $O\to\CC$ respectively.

\begin{theorem}\label{univAE}
There exists a unique $\Phi_{\cE_{I,O}}\in\Mor(\CC^O,\CC^I\tens\sA(\cE_{I,O}))$ such that for any \cst-algebra $\sB$ and any $\Psi\in\Mor(\CC^O,\CC^I\tens\sB)$ there is a unique $\Theta\in\Mor(\sA(\cE_{I,O}),\sB)$ such that $\Psi=(\id\tens\Theta)\comp\Phi_{\cE_{I,O}}$.
\end{theorem}

\begin{proof}
As we mentioned in the introduction, for any \cst-algebra $\sC$ a morphism from $\CC^O$ to $\sC$ is determined uniquely by a family of $|O|$ projections in $\M(\sC)$ summing up to $\I$. The elements
\[
\biggl\{\sum_{x\in{I}}e_x\tens{p_{x,a}}\biggr\}_{a\in{O}}
\]
form such a family for $\sC=\CC^I\tens\sA(\cE_{I,O})$ and hence there exists a unique $\Phi_{\cE_{I,O}}\in\Mor(\CC^O,\CC^I\tens\sA(\cE_{I,O})$ such that
\[
\Phi_{\cE_{I,O}}(e_a)=\sum_{x\in{I}}e_x\tens{p_{x,a}},\qqquad{a}\in{O}.
\]

Now let $\Psi\in\Mor(\CC^O,\CC^I\tens\sB)$ for some \cst-algebra $\sB$. Then for each $a\in{O}$ there exist elements $\{q_{x,a}\}_{x\in{I}}$ such that
\[
\Psi(e_a)=\sum_{x\in{I}}e_x\tens{q_{x,a}},\qqquad{a}\in{O}.
\]
Clearly each $q_{x,a}$ can be written as $(\delta_x\tens\id)\Psi(e_a)$, where $\delta_x$ is the evaluation functional on the algebra $\CC^I$ and one easily sees that each $q_{x,a}\in\M(\sB)$ is a projection. Hence, by definition of $\sA(\cE_{I,O})$, there exists a unique unital $*$-homomorphism $\Theta:\sA(\cE_{I,O})\to\M(\sB)$ such that $\Theta(p_{x,a})=q_{x,a}$ for all $(a,x)\in{I}\times{O}$. Obviously $\Psi=(\id\tens\Theta)\comp\Phi_{\cE_{I,O}}$.
\end{proof}

Theorem \ref{univAE} says that the morphism $\Phi_{\cE_{I,O}}\in\Mor(\CC^O,\CC^I\tens\sA(\cE_{I,O}))$ has the universal property of the \emph{quantum family of all maps from $I$ to $O$} as defined in \cite[Definition 3.1(2)]{qfam}.

\begin{corollary}
The \cst-algebra $\sA(\cE_{I,O})$ is canonically isomorphic to the \cst-algebra of continuous functions on the quantum space of all maps from $I$ to $O$. In particular $\sA(\cE_{I,O})$ is isomorphic to the free product $\underbrace{\CC^O*\dotsm*\CC^O}_{|I|}$ of $|I|$ copies of the abelian \cst-algebra $\CC^O$.
\end{corollary}

\begin{proof}
The first statement is true by definition, while the second is only a slight extension of the proof of \cite[Theorem 2.1]{qmqs}. Let $\sC=\underbrace{\CC^O*\dotsm*\CC^O}_{|I|}$ and for each $x\in{I}$ let $\iota_x$ be the embedding $\CC^O\hookrightarrow\sC$ onto the copy of $\CC^O$ corresponding to $x$.

Define $\Lambda:\CC^O\to\CC^I\tens\sC$ by
\[
\Lambda(v)=\sum_{x\in{I}}e_x\tens\iota_x(v),\qqquad{v}\in\CC^O
\]
(existence of such a $*$-homomorphism is proved along the lines indicated in the introduction).

Now let us take $\Psi\in\Mor(\CC^O,\CC^I\tens\sB)$. Then $\bigl\{(\delta_x\tens\id)\comp\Psi\bigr\}_{x\in{I}}$ is a family of unital $*$-homomorphisms $\CC^O\to\M(\sB)$ and so, by the universal property of the free product, there exists a unique unital $*$-homomorphism $\Theta:\sC\to\M(\sB)$ such that
\[
(\delta_x\tens\id)\comp\Psi=\Theta\comp\iota_x,\qqquad{x}\in{I}
\]
and hence $\Psi=(\id\tens\Theta)\comp\Lambda$. In other words $(\sC,\Lambda)$ has the universal property we have shown for $(\sA(\cE_{I,O}),\Phi_{\cE_{I,O}})$ in Theorem \ref{univAE}. It follows that there is an isomorphism $\Gamma:\sA(\cE_{I,O})\to\sC$ such that $\Lambda=(\id\tens\Gamma)\Phi_{\cE_{I,O}}$.
\end{proof}

\begin{remark}
The proof above derives the isomorphism $\sA(\cE_{I,O})\cong\underbrace{\CC^O*\dotsm*\CC^O}_{|I|}$ from the universal property of $(\sA(\cE_{I,O}),\Phi_{\cE_{I,O}})$, but it can just as well be proved using the construction of $\sA(\cG)$ for $\cG=\cE_{I,O}$ given in \cite{HMPS} Where $\sA(\cG)$ is defined as the quotient of the free product by an ideal which in case of $\cG=\cE_{I,O}$ is the zero ideal.
\end{remark}

In the case of a game $\cG$ more complicated than $\cE_{I,O}$ the corresponding \cst-algebra is still equipped with a map from $\CC^O$ to $\CC^I\tens\sA(\cG)$ which possesses a universal property. This universal property is, however, slightly more complicated to express. We have already used a couple of times the fact that given any \cst-algebra $\sB$ and any $\Psi\in\Mor(\CC^O,\CC^I\tens\sB)$ we can write the value of $\Psi$ on elements $\{e_a\}_{a\in{O}}$ of the standard basis $\CC^O$ as
\[
\Psi(e_a)=\sum_{x\in{I}}e_x\tens{q_{x,a}},\qqquad{a}\in{O}
\]
(with $\{e_x\}_{x\in{I}}$ the standard basis of $\CC^I$) and clearly each $q_{x,a}$ is a projection in $\M(\sB)$. Obviously, if the projections $\{q_{x,a}\}_{x\in{I},a\in{O}}$ satisfy
\[
q_{x,a}q_{y,b}=\lambda(a,b,x,y)q_{x,a}q_{y,b},\qqquad{x,y}\in{I},\:a,b\in{O}
\]
then by the universal property of $\sA(\cG)$ there exists a unique $\Theta\in\Mor(\sA(\cG),\sB)$ mapping $p_{x,a}$ onto $q_{x,a}$ for all $(a,x)\in{I}\times{O}$. Furthermore $\Psi=(\id\tens\Theta)\comp\Phi_\cG$, where $\Phi_\cG\in\Mor(\CC^O,\CC^I\tens\sA(\cG))$ is defined by
\[
\Phi(e_a)=\sum_{x\in{I}}e_x\tens{p_{x,a}},\qqquad{a}\in{O}.
\]
(note that the last map can be expressed as the composition $(\id\tens\pi_\cG)\comp\Psi_{\cE_{I,O}}$, where $\pi_\cG$ is the quotient map from $\sA(\cE_{I,O})$ to $\sA(\cG)$).

As noted in the proof of Theorem \ref{univAE}, for any $\Psi\in\Mor(\CC^O,\CC^I\tens\sB)$ the element $q_{x,a}$ as above can be described as $(\delta_x\tens\id)\Psi(e_a)$, where $\delta_x$ is the evaluation functional on $\CC^I$ treated in a natural way as the algebra of all functions $O\to\CC$. Using this and the leg-numbering notation (see e.g.~\cite[p.~253]{qE2}) we can express the universal property of $(\sA(\cG),\Phi_\cG)$ as follows:

\begin{proposition}
Given a \cst-algebra $\sB$ and $\Psi\in\Mor(\CC^O,\CC^I\tens\sB)$ such that
\[
(\delta_x\tens\delta_y\tens\id)\bigl((\Psi(e_a)_{13}\Psi(e_b)_{23})\bigr)=\lambda(a,b,x,y)(\delta_x\tens\delta_y\tens\id)\bigl((\Psi(e_a)_{13}\Psi(e_b)_{23})\bigr)
\]
for all $x,y\in{I}$ and $a,b\in{O}$ there exists a unique $\Theta\in\Mor(\sA(\cG),\sB)$ such that $\Psi=(\Theta\tens\id)\comp\Phi_\cG$.
\end{proposition}

\begin{proof}
Obvious.
\end{proof}

\begin{remark}\label{genwv}
It is immediate from the definitions of $\sA(\cG)$ and $\Phi_\cG$ that the \cst-algebra $\sA(\cG)$ is generated by the set
\[
\bigl\{(\omega\tens\id)\Phi_\cG(v)\st{v}\in\CC^O,\:\omega\in(\CC^I)^*\bigr\}.
\]
\end{remark}

\begin{example}\label{grHom}
One of the many examples of a two-person finite input-output game is the ``graph homomorphism game''. We are given two finite graphs $G_1=(V_1,E_1)$ and $G_2=(V_2,E_2)$ and we put $I=V_1$, $O=V_2$. Then define
\[
\lambda(a,b,x,y)=\begin{cases}
0&(x,y)\in{E_1},\:(a,b)\not\in{E_2},\\
1&\text{else}.
\end{cases}
\]
The interpretation of this game is that Alice and Bob win if they manage to convince the referee that they have come up with a graph homomorphism $G_1\to{G_2}$ and given respective inputs $x,y\in{V_1}$ their answers $a,b\in{V_2}$ are the images of $x$ and $y$ under the presumed homomorphism.

It is not difficult to see that the ``classical points'' (or \emph{Gelfand spectrum}) of $\sA(\cG)$ is in a natural bijection with the set of all graph homomorphisms $G_1\to{G_2}$ (cf.~\cite[Theorem 4.4]{qfam}).
\end{example}

\section{Synchronous games and quantum semigroups}\label{SynchGamesQS}

The \cst-algebras associated with some synchronous games posses additional interesting properties -- particularly in the special case when the input and output sets are the same and the rules of the game satisfy an additional condition. Therefore in this section we will assume that $O=I$ and that the rules -- now given by a function $\lambda:I^4\to\{0,1\}$ -- satisfy
\begin{equation}\label{star}
\Bigl(\;\lambda(i,j,k,l)=0\;\Bigr)\Longrightarrow\Bigl(\;\forall\;{r,s}\in{I}\:\:\lambda(i,j,r,s)\lambda(r,s,k,l)=0\;\Bigr),\qqquad{i,j,k,l}\in{I}.
\end{equation}
Note that the ``graph homomorphism game'' of Example \ref{grHom} satisfies this condition (in case $G_1=G_2$, of course). Indeed, denoting by $E$ the set of edges of $G=G_1=G_2$ we have $\lambda(i,j,k,l)=0$ if and only if $(k,l)\in{E}$ and $(i,j)\not\in{E}$. Then if $(r,s)$ is any pair of vertices of $G$ then either $(r,s)\in{E}$, in which case $\lambda(i,j,r,s)=0$ or $(r,s)\not\in{E}$, in which case $\lambda(r,s,k,l)=0$.

\begin{theorem}\label{semi}
Let $\cG$ be a synchronous game with $O=I$ satisfying condition \eqref{star}. Then
\begin{enumerate}
\item\label{semi1} there exists a unique $\Delta_\cG\in\Mor(\sA(\cG),\sA(\cG)\tens\sA(\cG))$ such that
\begin{equation}\label{PhiDel}
(\Phi_\cG\tens\id)\comp\Phi_\cG=(\id\tens\Delta_\cG)\comp\Phi_\cG,
\end{equation}
\item\label{semi2} $\Delta_\cG$ is coassociative: $(\Delta_\cG\tens\id)\comp\Delta_\cG=(\id\tens\Delta_\cG)\comp\Delta_\cG$ and consequently endows the quantum space underlying $\sA(\cG)$ with the structure of a compact quantum semigroup; moreover $\Phi_\cG$ is an action of this quantum semigroup on $\CC^I$;
\item\label{semi3} if $\mathbb{S}$ is a quantum semigroup with an action $\Psi\in\Mor(\CC^I,\CC^I\tens\C_0(\mathbb{S}))$ which satisfies the condition
\[
(\delta_k\tens\delta_l\tens\id)\bigl((\Psi(e_i)_{13}\Psi(e_j)_{23})\bigr)=\lambda(i,j,k,l)(\delta_k\tens\delta_l\tens\id)\bigl((\Psi(e_i)_{13}\Psi(e_j)_{23})\bigr)
\]
for all $i,j,k,l\in{I}$ then the unique $\Theta\in\Mor(\sA(\cG),\C(\mathbb{S}))$ such that $\Psi=(\id\tens\Theta)\comp\Phi_\cG$ satisfies $\Delta_\mathbb{S}\comp\Theta=(\Theta\tens\Theta)\comp\Delta_\cG$.
\end{enumerate}
\end{theorem}

\begin{proof}
Ad \eqref{semi1}. Consider the \cst-algebra $\sB=\sA(\cG)\tens\sA(\cG)$ and $\Psi\in\Mor(\CC^O,\CC^I\tens\sB)$ defined as $\Psi=(\Phi_\cG\tens\id)\comp\Phi_\cG$. For $i,j\in{I}$ we have
\[
\Psi(e_i)=\sum_{r\in{I}}\Phi_\cG(e_r)\tens{p_{r,i}}=\sum_{r,k\in{I}}e_k\tens{p_{k,r}}\tens{p_{r,i}}
=\sum_{k\in{I}}e_k\tens{Q_{k,i}},
\]
where $Q_{k,i}=\sum\limits_{r\in{I}}p_{k,r}\tens{p_{r,i}}$. Then using the defining relations of $\sA(\cG)$ we find that
\[
Q_{k,i}Q_{l,j}=\sum_{r,s\in{I}}p_{k,r}p_{l,s}\tens{p_{r,i}p_{s,j}}
=\sum_{r,s\in{I}}\lambda(r,s,k,l)\lambda(i,j,r,s)p_{k,r}p_{l,s}\tens{p_{r,i}p_{s,j}}=0
\]
due to condition \eqref{star}. It follows that there exists a unique $\Delta_\cG\in\Mor(\sA(\cG),\sA(\cG)\tens\sA(\cG))$ such that $\Psi=(\id\tens\Delta_\cG)\comp\Phi_\cG$ which is exactly \eqref{PhiDel}.

Ad \eqref{semi2}. Only the first statement requires a proof. Using repeatedly property \eqref{PhiDel} we obtain
\[
\bigl((\id\tens[(\Delta_\cG\tens\id)\comp\Delta_\cG]\bigr)\comp\Phi_\cG
=(\Phi_\cG\tens\id\tens\id)\comp(\Phi_\cG\tens\id)\comp\Phi_\cG
=\bigl((\id\tens[(\id\tens\Delta_\cG)\comp\Delta_\cG]\bigr)\comp\Phi_\cG.
\]
Now applying both sides to any $v\in\CC^O$ and computing values of $(\omega\tens\id\tens\id\tens\id)$ of both sides (with arbitrary $\omega\in(\CC^I)^*$) we find that
\[
\bigl((\Delta_\cG\tens\id)\comp\Delta_\cG)\bigr)(x)=\bigl((\id\tens\Delta_\cG)\comp\Delta_\cG\bigr)(x)
\]
for elements $x$ of the form $x=(\omega\tens\id)\Phi_\cG(v)$. Equality $(\Delta_\cG\tens\id)\comp\Delta_\cG=(\id\tens\Delta_\cG)\comp\Delta_\cG$ follows now from Remark \ref{genwv}.

Ad \eqref{semi3}. This is an almost literal repetition of the proof of \cite[Theorem 4.7]{qfam}. One checks that $\Theta$ must satisfy
\[
\bigl(\id\tens[(\Theta\tens\Theta)\comp\Delta_\cG)]\bigr)\comp\Phi_\cG=\bigl(\id\tens[\Delta_\mathbb{S}\comp\Theta]\bigr)\comp\Phi_\cG
\]
and then uses Remark \ref{genwv}.
\end{proof}

Let $\mathbb{S}$ be a quantum semigroup. A \emph{counit} for $\mathbb{S}$ is a character $\eps$ of $\C_0(\mathbb{S})$ such that $(\eps\tens\id)\comp\Delta_\mathbb{S}=(\id\tens\eps)\comp\Delta_\mathbb{S}=\id$.

It is interesting that the quantum semigroup described by $(\sA(\cG),\Delta_\cG)$ might in some situations not admit a counit.

\begin{theorem}\label{coUnit}
Let $\cG$ be a synchronous game with $O=I$ and rules $\lambda$ satisfying condition \eqref{star}. Consider the following conditions:
\begin{enumerate}
\item\label{coUnit1} the the quantum semigroup described by $(\sA(\cG),\Delta_\cG)$ admits a counit,
\item\label{coUnit2} there exists a character $\eps$ of $\sA(\cG)$ such that $(\eps\tens\id)\comp\Phi_\cG=\id$,
\item\label{coUnit3} for each $i,j\in{I}$ we have $\lambda(i,j,i,j)=1$.
\end{enumerate}
Then \eqref{coUnit3} $\Leftrightarrow$ \eqref{coUnit2} $\Rightarrow$ \eqref{coUnit1}.
\end{theorem}

\begin{proof}
\eqref{coUnit3} $\Rightarrow$ \eqref{coUnit2}
If $\lambda(i,j,i,j)=1$ for each pair $(i,j)$ then the collection of numbers $\{\delta_{i,j}\}_{i,j\in{I}}$ satisfies
\begin{itemize}
\item $\sum\limits_{j\in{I}}\delta_{i,j}=1$ for all $i\in{I}$,
\item if $\lambda(i,j,k,l)=0$ then $\delta_{k,i}\delta_{l,j}=0$.
\end{itemize}
It follows that there exists a unital $*$-homomorphism $\eps:\sA(\cG)\to\CC$ (in other words, a character of $\sA(\cG)$) such that
\[
\eps(p_{i,j})=\delta_{i,j},\qqquad{i,j}\in{I}.
\]
Obviously $\eps$ satisfies $(\eps\tens\id)\comp\Phi_\cG=\id$.

\eqref{coUnit2} $\Rightarrow$ \eqref{coUnit3}. Clearly a character $\eps$ of $\sA(\cG)$ which satisfies $(\eps\tens\id)\comp\Phi_\cG=\id$ must assign the value $\delta_{i,j}$ to $p_{i,j}$, hence $\eps(p_{i,i})\eps(p_{j,j})=1$, so we cannot have $\lambda(i,j,i,j)=0$ for any $i,j$.

\eqref{coUnit2} $\Rightarrow$ \eqref{coUnit1}. It follows easily from the property $(\eps\tens\id)\comp\Phi_\cG=\id$ that
\[
\bigl(\id\tens[(\eps\tens\id)\comp\Delta_\cG]\bigr)\comp\Phi_\cG=\Phi_\cG
=\bigl(\id\tens[(\id\tens\eps)\comp\Delta_\cG]\bigr)\comp\Phi_\cG,
\]
so taking slices with $\omega\in(\CC^I)^*$ over the first leg and using Remark \ref{genwv} we obtain
\[
(\eps\tens\id)\comp\Delta_\cG=(\id\tens\eps)\comp\Delta_\cG.
\]
\end{proof}

\begin{remark}
\noindent
\begin{enumerate}
\item One can easily supplement Theorem \ref{coUnit} with the implication \eqref{coUnit1} $\Rightarrow$ \eqref{coUnit2} under the additional assumption that $\CC^I$ is generated by the set
\[
\bigl\{(\id\tens\nu)\Phi_\cG(v)\st{v}\in\CC^I,\:\nu\in\sA(\cG)^*\bigr\}.
\]
Indeed, if $\eps$ is a character of $\sA(\cG)$ such that $(\eps\tens\id)\comp\Delta_\cG=\id$ then
\[
\Phi_\cG=\bigl(\id\tens[(\eps\tens\id)\comp\Delta_\cG]\bigr)\comp\Phi_\cG=\Bigl([(\id\tens\eps)\comp\Phi_\cG]\tens\id\bigr)\comp\Phi_\cG,
\]
so slicing with $\nu\in\sA(\cG)^*$ over the last leg we obtain $(\id\tens\eps)\comp\Phi_\cG=\id$.
\item Note also that condition \eqref{coUnit3} of Theorem \ref{coUnit} holds for the ``graph homomorphism game'' of Example \ref{grHom} with $G_1=G_2$.
\end{enumerate}
\end{remark}

\section{Quantum groups}

If $\cG$ is the graph homomorphism game with $G_1=G_2=G$ as considered at the beginning of Section \ref{SynchGamesQS} then one can introduce a positive functional $\omega$ on $\CC^I$ (here $I=V$ is the set of vertices of $G$) corresponding to the the measure assigning mass $1$ to all points of $I$, i.e.~$\omega(e_i)=1$ for all $i\in{I}$. Clearly graph homomorphisms $G\to{G}$ preserving this measure are automorphisms of $G$.

Let $\mathbb{S}$ be a quantum semigroup acting on $I$ via $\Psi\in\Mor(\CC^I,\CC^I\tens\C_0(\mathbb{S}))$ and let us write
\[
\Psi(e_j)=\sum_{i\in{I}}e_i\tens{q_{i,j}},\qqquad{j}\in{I}.
\]
Then the condition that $\omega$ be preserved by the action $\Psi$ translates into relations
\[
\sum_{i\in{I}}q_{i,j}=\I,\qqquad{j}\in{I}
\]
on the elements $\{q_{i,j}\}_{i,j\in{I}}$.

Now if $\cG$ is a synchronous game with $O=I$ and rules satisfying \eqref{star} then adding analogous relations to the list defining $\sA(\cG)$ we obtain the universal quantum family of maps $I\to{I}$ which preserves $\omega$ and is compatible with $\lambda$.

\begin{proposition}\label{invw}
Let $\cG$ be a synchronous game with $O=I$. Let $\widetilde{\sA}(\cG)$ be the universal \cst-algebra generated by projections $\{\tilde{p}_{i,j}\}_{i,j\in{I}}$ such that
\begin{equation}\label{moreRels}
\begin{aligned}
\sum_{j\in{I}}\tilde{p}_{i,j}&=\I,&\qquad&i\in{I},\\
\sum_{i\in{I}}\tilde{p}_{i,j}&=\I,&\qquad&j\in{I},\\
\tilde{p}_{k,i}\tilde{p}_{l,j}&=\lambda(i,j,k,l)\tilde{p}_{k,i}\tilde{p}_{l,j},&\qquad&i,j,k,l\in{I}.
\end{aligned}
\end{equation}
Then there exists a unique $\widetilde{\Phi}_\cG\in\Mor(\CC^I,\CC^I\tens\widetilde{\sA}(\cG))$ such that for any \cst-algebra $\sB$ and any $\Psi\in\Mor(\CC^I,\CC^I\tens\sB)$ such that
\begin{equation}\label{compat}
\begin{array}{l}
(\delta_k\tens\delta_l\tens\id)\bigl((\Psi(e_i)_{13}\Psi(e_j)_{23})\bigr)\\
\qquad\qquad=\lambda(i,j,k,l)(\delta_k\tens\delta_l\tens\id)\bigl((\Psi(e_i)_{13}\Psi(e_j)_{23})\bigr),
\end{array}\qqquad{i,j,k,l}\in{I}
\end{equation}
and
\[
(\omega\tens\id)\Psi((v)=\omega(v)\I,\qqquad{v}\in\CC^I
\]
then there exists a unique $\Theta\in\Mor(\widetilde{\sA}(\cG),\sB)$ such that $\Psi=(\id\tens\Theta)\comp\widetilde{\Phi}_\cG$.

Moreover $\widetilde{\Phi}_\cG$ is a quantum family of invertible maps in the sense of \cite[Definition 3.1]{invert} preserving the state $\omega$ on $\CC^I$.
\end{proposition}

\begin{proof}
The \cst-algebra $\widetilde{\sA}(\cG)$ is clearly a quotient of $\sA(\cG)$ and writing $\tilde{\pi}$ for the quotient map we have $\tilde{p}_{i,j}=\tilde{\pi}(p_{i,j})$ for all $i,j\in{I}$. We can define $\widetilde{\Phi}_\cG$ as the composition of $(\id\tens\tilde{\pi})\comp\Phi_\cG$. Then if $(B,\Psi)$ are as in the statement of the theorem then, as we have seen before, we have
\[
\Psi(e_j)=\sum_{i\in{I}}e_i\tens{q_{i,j}},\qqquad{j}\in{I}
\]
and the elements $\{q_{i,j}\}_{i,j\in{I}}$ satisfy relations analogous to \eqref{moreRels}. Therefore there exists a unique $\Theta\in\Mor(\widetilde{\sA}(\cG),\sB)$ sending $\tilde{p}_{i,j}$ to $q_{i,j}$ for all $i,j\in{I}$. This is precisely the condition $\Psi=(\id\tens\Theta)\comp\widetilde{\Phi}_\cG$.

Note that by orthogonality of $\{\tilde{p}_{k,j}\}_{k\in{I}}$ for each $j$ we have for any fixed $i\in{I}$
\begin{align*}
\sum_{j\in{I}}\widetilde{\Phi}_\cG(e_j)(\I\tens\tilde{p}_{i,j})
&=\sum_{j\in{I}}\sum_{k\in{I}}(e_k\tens\tilde{p}_{k,j})(\I\tens\tilde{p}_{i,j})\\
&=\sum_{j\in{I}}\sum_{k\in{I}}e_k\tens\delta_{i,k}\tilde{p}_{i,j}=e_i\tens\sum_{j\in{I}}\tilde{p}_{i,j}=e_i\tens\I.
\end{align*}
It follows that the subspace
\[
\operatorname{span}\bigl\{\widetilde{\Phi}_\cG(v)(\I\tens{a}){\bigl.\st\bigr.}{v}\in\CC^I,\:a\in\widetilde{\sA}(\cG)\bigr\}
\]
is dense in $\CC^I\tens\widetilde{\sA}(\cG)$, i.e.~$\widetilde{\Phi}_\cG$ is a quantum family of invertible maps. The fact that $\widetilde{\Phi}_\cG$ preserves $\omega$ is obvious.
\end{proof}

We will say that a quantum family of maps $\Psi\in\Mor(\CC^I,\CC^I\tens\sB)$ is \emph{compatible with the game $\cG$} if it satisfies condition \eqref{compat} of Proposition \ref{invw}. Clearly $(\widetilde{\sA}(\cG),\widetilde{\Phi}_\cG)$ is a universal quantum family of maps which are compatible with $\cG$ and preserve $\omega$. It is also easy to see that in case the rules $\lambda$ of $\cG$ satisfy \eqref{star} then this universal family has additional structure and properties. We formulate these in the next theorem.

\begin{theorem}\label{cqgG}
Let $\cG$ be a synchronous game with $O=I$ and rules satisfying condition \eqref{star}. Then
\begin{enumerate}
\item there exists a unique $\widetilde{\Delta}_\cG\in\Mor(\widetilde{\sA}(\cG),\widetilde{\sA}(\cG)\tens\widetilde{\sA}(\cG))$ such that
\[
(\widetilde{\Phi}_\cG\tens\id)\comp\widetilde{\Phi}_\cG=(\id\tens\widetilde{\Delta}_\cG)\comp\widetilde{\Phi}_\cG,
\]
\item\label{cqgG2} $(\widetilde{\sA}(\cG),\widetilde{\Delta}_\cG)$ defines a compact quantum group $\GG$,
\item for any compact quantum group $\KK$ and any action of $\KK$ on $I$ given by
\[
\Psi\in\Mor(\CC^I,\CC^I\tens\C(\KK))
\]
which is compatible with the game $\cG$ and preserves the state $\omega$ there exists a unique $\Theta\in\Mor(\widetilde{\sA}(\cG),\C(\KK)\bigr)$ such that
\begin{equation}\label{ThetaPsi}
\Psi=(\id\tens\Theta)\comp\widetilde{\Phi}_\cG;
\end{equation}
moreover $\Theta$ is a morphism of compact quantum groups.
\end{enumerate}
\end{theorem}

\begin{proof}
The quantum family $\widetilde{\Phi}_\cG\in\Mor(\CC^I,\CC^I\tens\widetilde{\sA}(\cG))$ preserves the state $\omega$ and hence so does $(\widetilde{\Phi}_\cG\tens\id)\comp\widetilde{\Phi}_\cG$:
\begin{align*}
(\omega\tens\id\tens\id)(\widetilde{\Phi}_\cG\tens\id)\widetilde{\Phi}_\cG(v)
&=\I\tens\bigl((\omega\tens\id)\widetilde{\Phi}_\cG(v)\bigr)\\
&=\omega(v)(\I\tens\I).
\end{align*}
Together with condition \eqref{star} this guarantees that $(\widetilde{\Phi}_\cG\tens\id)\comp\widetilde{\Phi}_\cG$ is of the form $(\id\tens\widetilde{\Delta}_\cG)\comp\widetilde{\Phi}_\cG$ for a unique $\widetilde{\Delta}_\cG\in\Mor(\widetilde{\sA}(\cG),\widetilde{\sA}(\cG)\tens\widetilde{\sA}(\cG))$. It is also easy to see that on generators $\{\tilde{p}_{i,j}\}$ it acts in the standard way:
\[
\widetilde{\Delta}_\cG(\tilde{p}_{i,j})=\sum_{k\in{I}}\tilde{p}_{i,k}\tens\tilde{p}_{k,j},\qqquad{i,j}\in{I}.
\]
Hence $\widetilde{\Delta}_\cG$ is coassociative. The fact that $(\widetilde{\sA}(\cG),\widetilde{\Delta}_\cG)$ defines a compact quantum group follows now either from general facts about quantum families of invertible maps (\cite[Section 4]{invert}), or the calculations
\begin{align*}
\sum_{k\in{I}}\widetilde{\Delta}_\cG(\tilde{p}_{i,k})(\I\tens\tilde{p}_{j,k})&=\sum_{k,l\in{I}}\tilde{p}_{i,l}\tens\tilde{p}_{l,k}\tilde{p}_{j,k}\\
&=\sum_{k,l\in{I}}\tilde{p}_{i,l}\tens\delta_{l,j}\tilde{p}_{j,k}=\sum_{k\in{I}}\tilde{p}_{i,j}\tens\tilde{p}_{j,k}=\tilde{p}_{i,j}\tens\I
\end{align*}
and
\begin{align*}
\sum_{k\in{I}}(\tilde{p}_{k,i}\tens\I)\widetilde{\Delta}_\cG(\tilde{p}_{k,j})&=\sum_{k,l\in{I}}\tilde{p}_{k,i}\tilde{p}_{k,l}\tens\tilde{p}_{l,j}\\
&=\sum_{k,l\in{I}}\delta_{i,l}\tilde{p}_{k,i}\tens\tilde{p}_{l,j}=\sum_{k\in{I}}\tilde{p}_{k,i}\tens\tilde{p}_{i,j}=\I\tens\tilde{p}_{i,j}
\end{align*}
which show that
\[
\operatorname{span}\bigl\{\widetilde{\Delta}_\cG(a)(\I\tens{b}){\bigl.\st\bigr.}a,b\in\widetilde{\sA}(\cG)\bigr\}\quad\text{and}\quad
\operatorname{span}\bigl\{(a\tens\I)\widetilde{\Delta}_\cG(b){\bigl.\st\bigr.}a,b\in\widetilde{\sA}(\cG)\bigr\}
\]
are dense in $\widetilde{\sA}(\cG)\tens\widetilde{\sA}(\cG)$ (see \cite[Definition 2.1]{cqg}).

Given an action $\Psi\in\Mor(\CC^I,\CC^I\tens\C(\KK))$ of a compact quantum group $\KK$ which is compatible with $\cG$ and preserves $\omega$ the morphism $\Theta\in\Mor(\widetilde{\sA}(\cG),\C(\KK))$ satisfying \eqref{ThetaPsi} exists by universal property of $(\widetilde{\sA}(\cG),\widetilde{\Phi}_\cG)$. The fact that $\Theta$ is a morphisms of compact quantum groups (i.e.~preserves comultiplications) can be proved in the same way as in Theorem \ref{semi}\eqref{semi3}.
\end{proof}

Let $G$ be a finite graph and let $\cG$ be the corresponding graph endomorphism game. Then the Gelfand transform is an epimorphism of $\widetilde{\sA}(\cG)$ onto the algebra of functions on the group of automorphisms of $G$. In other words the classical points of the compact quantum group $\GG$ described in Theorem \ref{cqgG}\eqref{cqgG2} form the classical automorphism group of $G$. Indeed, if $\chi$ is a character of $\widetilde{\sA}(\cG)$ and $r_{k,l}=\chi(p_{k,l})$ for all $k,l\in{I}$ then
\begin{itemize}
\item each $r_{i,j}$ is either $0$ or $1$,
\item for each $i$ there is precisely one $j$ such that $r_{i,j}\neq{0}$,
\item for each $j$ there is precisely one $i$ such that $r_{i,j}\neq{0}$.
\end{itemize}
It follows that $\chi$ determines a unique permutation $\pi_\chi$ of the vertices of $G$. Moreover, compatibility with the rules of $\cG$ implies that if two vertices $i$ and $j$ are connected by an edge then so are $\pi_\chi(i)$ and $\pi_\chi(j)$. In other words $\pi_\chi$ belongs to the group of permutations of $G$. Conversely, any automorphism $\varphi$ of $G$ gives rise to a character of $\widetilde{\sA}(\cG)$ by sending $p_{i,j}$ to $1$ if $\varphi(i)=j$ and to $0$ otherwise.

This shows that the quantum group $\GG$ described in Theorem \ref{cqgG} is precisely the one discussed in \cite{banica} for small metric spaces arising from graphs. By dividing out certain additional relations one obtains the quantum automorphism group of the same graph described in \cite[Section 3]{jb}.

\section{acknowledgments}
The author wishes to thank the referee for very helpful comments and suggestion of several references. The research presented in this paper was partially supported by NCN (National Science Center, Poland) grant no.~2015/17/B/ST1/00085.

\end{document}